\documentclass[12pt]{amsart}
\pdfoutput=1 
\usepackage[english]{babel}
\usepackage{graphicx}
\usepackage{latexsym,amssymb,amsfonts,amsmath, amscd, euscript, MnSymbol}
\usepackage{enumitem}
\usepackage{url}

\newif\ifcard
\cardfalse

\usepackage{float}
\usepackage{layout}
\usepackage[top=3cm, bottom=3cm, left=2.5cm, right=2.5cm]{geometry}

\newcommand{\N}{\mathbb{N}}

\newcommand\pow{\mathbb P}

\newcommand\dc{\mathop{\downarrow}\nolimits}
\newcommand\upc{\mathop{\uparrow}\nolimits}

\newtheorem{theorem}{Theorem}

\newtheorem{claim}[theorem]{Claim}

\newtheorem{lemma}{Lemma}

\newtheorem{proposition}[theorem]{Proposition}
\theoremstyle{remark}
\newtheorem{remark}[theorem]{Remark}

\def\Power #1 { \pow (#1) }
\def\Bidom #1 { {\mathfrak P} (#1) }

\def\centerpicture #1 by #2 (#3){\leavevmode
        \vbox to #2{
        \hrule width #1 height 0pt depth 0pt
        \vfill
        \special{pictfile #3}}}
\title[Topological spaces with no infinite discrete subspace]{A Few Characterizations  of Topological Spaces with No Infinite Discrete Subspace}

\author[J. Goubault-Larrecq]{Jean Goubault-Larrecq}
\address{Universit\'e Paris-Saclay, CNRS, ENS Paris-Saclay, Laboratoire M\'ethodes Formelles, 91190, Gif-sur-Yvette, France.}
\email{goubault@lsv.fr}

\author[M. Pouzet]{Maurice Pouzet}
\address{Univ. Lyon, Universit\'e Claude-Bernard Lyon1, CNRS UMR 5208, Institut Camille Jordan,  43 bd. 11 Novembre 1918, 69622 Villeurbanne Cedex, France and Mathematics \& Statistics Department, University of Calgary, Calgary, Alberta, Canada T2N 1N4}
\email{pouzet@univ-lyon1.fr, mpouzet@ucalgary.ca }

\date{\today}

\keywords{Ordered set, closure system, Noetherian topological space, well-quasi-order}
\subjclass[2010]{Primary 
54G99; 
Secondary 
06A07 
}
\thanks{The first author was supported by grant ANR-17-CE40-0028 of the French National Research Agency ANR (project BRAVAS)}

\begin{document}

\begin{abstract}
  We give several characteristic properties of FAC spaces, namely
  topological spaces with no infinite discrete subspace. The first one
  was obtained in 2019 by the first author, and states that every
  closed set is a finite union of irreducible closed subsets.  The
  full result extends well-known characterizations of posets with no
  infinite antichain.  One of them is that FAC spaces are,
  equivalently, topological spaces in which every closed set contains
  a dense Noetherian subspace, or spaces in which every Hausdorff
  subspace is finite, or in which no subspace has any infinite
  relatively Hausdorff subset.  The latter comes with a nice min-max
  property, extending an observation of Erd\H os and Tarski in the
  case of posets: on spaces with no infinite relatively Hausdorff
  subset, the cardinalities of relatively Hausdorff subsets are
  bounded, and the least upper bound is also the least cardinality of
  a family of closed irreducible subsets that cover the space.
\end{abstract}

\maketitle

\noindent
\begin{minipage}{0.25\linewidth}
  \includegraphics[scale=0.2]{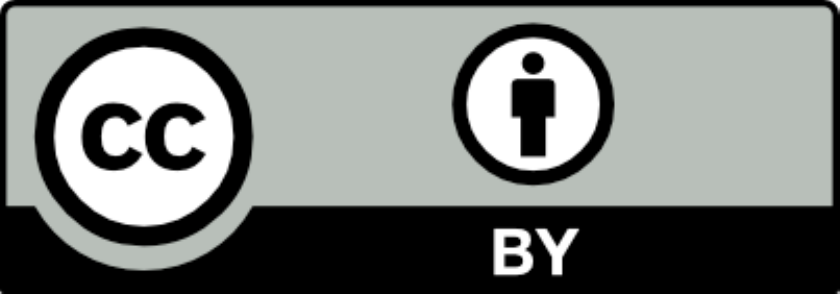}
\end{minipage}
\begin{minipage}{0.74\linewidth}
  \scriptsize
  For the purpose of Open Access, a CC-BY public copyright licence has
  been applied by the authors to the present document and will be
  applied to all subsequent versions up to the Author Accepted
  Manuscript arising from this submission.
\end{minipage}

\section{Introduction}

A topological space $T:=(E, \mathcal F)$, where $\mathcal F$ is the
set of closed subsets, is \emph{Noetherian} if every descending
sequence of closed subspaces is stationary. A subset $X$ of $E$ is
\emph{discrete} if and only if the induced topology on $X$ is the
discrete topology, namely if every subset of $X$ is closed with
respect to the induced topology. A closed subset is \emph{irreducible}
if it is non-empty and not the union of two proper closed subsets.
Our main objective is to show the following.

\begin{theorem}
  \label{theorem-main}
  The following properties are equivalent for a topological space
  $T:=(E, \mathcal F)$.
  \begin{enumerate}[label=(\roman*)]
  \item No infinite subset of $E$ is discrete;
  \item Every closed set is a finite union of irreducible closed subsets;
  \item Every closed set contains a dense subset on which the induced
    topology is Noetherian.
  \end{enumerate}
\end{theorem}

The equivalence between $(i)$ and $(ii)$ was proved by the first
author in 2019 \cite{goubault-larrecq}.  The resulting spaces were
called FAC spaces there, since they are a topological generalization
of posets with the \emph{f}inite \emph{a}ntichain \emph{c}ondition,
namely those that have no infinite antichain.  Hence, compared to
\cite{goubault-larrecq}, $(iii)$ is a new, equivalent definition of
FAC spaces, itself inspired from a well-known equivalent
characterization of posets with the finite antichain condition (see
Section~\ref{sec:posets}).  We will also mention a few other
equivalent conditions in Theorem~\ref{thm:haus:discr}, which, as we
will see, are connected to a nice min-max property, which we will
state in Proposition~\ref{prop:ET:top}.

We give a proof of the equivalence of $(iii)$ with $(i)$ and $(ii)$ in
Section~\ref{sec:proof}.  That also gives an alternative proof of the
equivalence between $(i)$ and $(ii)$, and will stress the importance
of the notion of infinite separating chain of closed sets, inspired
from work by Chakir and Pouzet \cite{chakir-pouzet,chakir}.  We make
additional remarks in Section~\ref{sec:remarks}, relating the result
to other characterizations of FAC spaces, to Noetherian spaces, to a
related result of Galvin, Milner and Pouzet in the larger context of
closure operators, to lattice properties, and finally to a min-max
result due to Erd\H os and Tarski \cite{erdos-tarski}, which will lead
us to the topological min-max theorem announced above.

\section{The proof}
\label{sec:proof}

We mimic the proof of a similar, fairly well-known result for posets,
which we will give in Section~\ref{sec:posets}. The significant part
is the implication $(i)\Rightarrow (iii)$.

We recall that a \emph{closure system} is a pair $(E, \varphi)$ where
$\varphi$ (the \emph{closure}) is a map from the power set of $E$ into
itself which is extensive, order-preserving and idempotent
\cite[Section~II.1]{Cohn:alg:1981}.  A closure is \emph{topological}
if and only if it commutes with finite unions.  The fixed points of a
topological closure form the lattice of closed sets of a topology, and
conversely, the usual closure operator of a topological space is a
topological closure.

Given a closure system $(E, \varphi)$, a subset $C\subseteq E$ is
\emph{closed} if $\varphi(C)=C$; it is \emph{independent} if
$x\not \in \varphi(C\setminus \{x\})$ for every $x\in C$; it is
\emph{generating} if $\varphi(C)= E$.  When $\varphi$ is a topological
closure, a set is independent if and only if it is discrete, and the
generating sets are called \emph{dense}.

The closure $\varphi_{\restriction E'}$ \emph{induced by $\varphi$} on
a subset $E'$ of $E$ is defined by
$\varphi_{\restriction E'}(X):= \varphi(X)\cap E'$ for every
$X\subseteq E'$.  \ifcard The \emph{dimension} of the closure system
is the least cardinal $\kappa$ for which there is a generating set of
cardinal $\kappa$.  \fi 

Although our results will only apply to topological closures, we use
the language of (general) closures so as to be able to relate our
results to a theorem of Galvin, Milner and Pouzet in
Section~\ref{sec:related-result}.

\begin{figure}
  \centering
  \begin{tabular}{cc}
    \multicolumn{2}{c}{\includegraphics[scale=0.15]{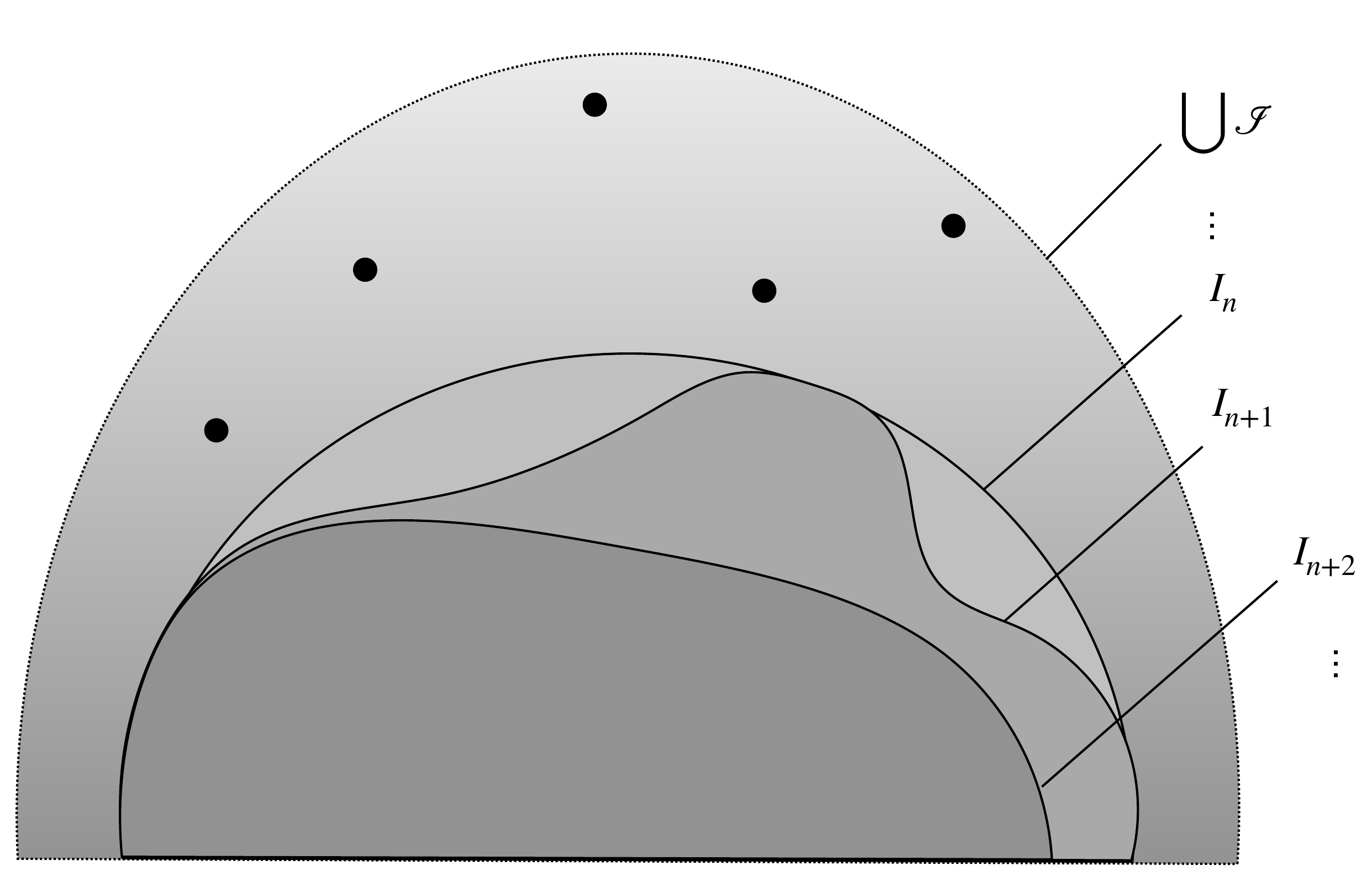}}
    \\
    \includegraphics[scale=0.15]{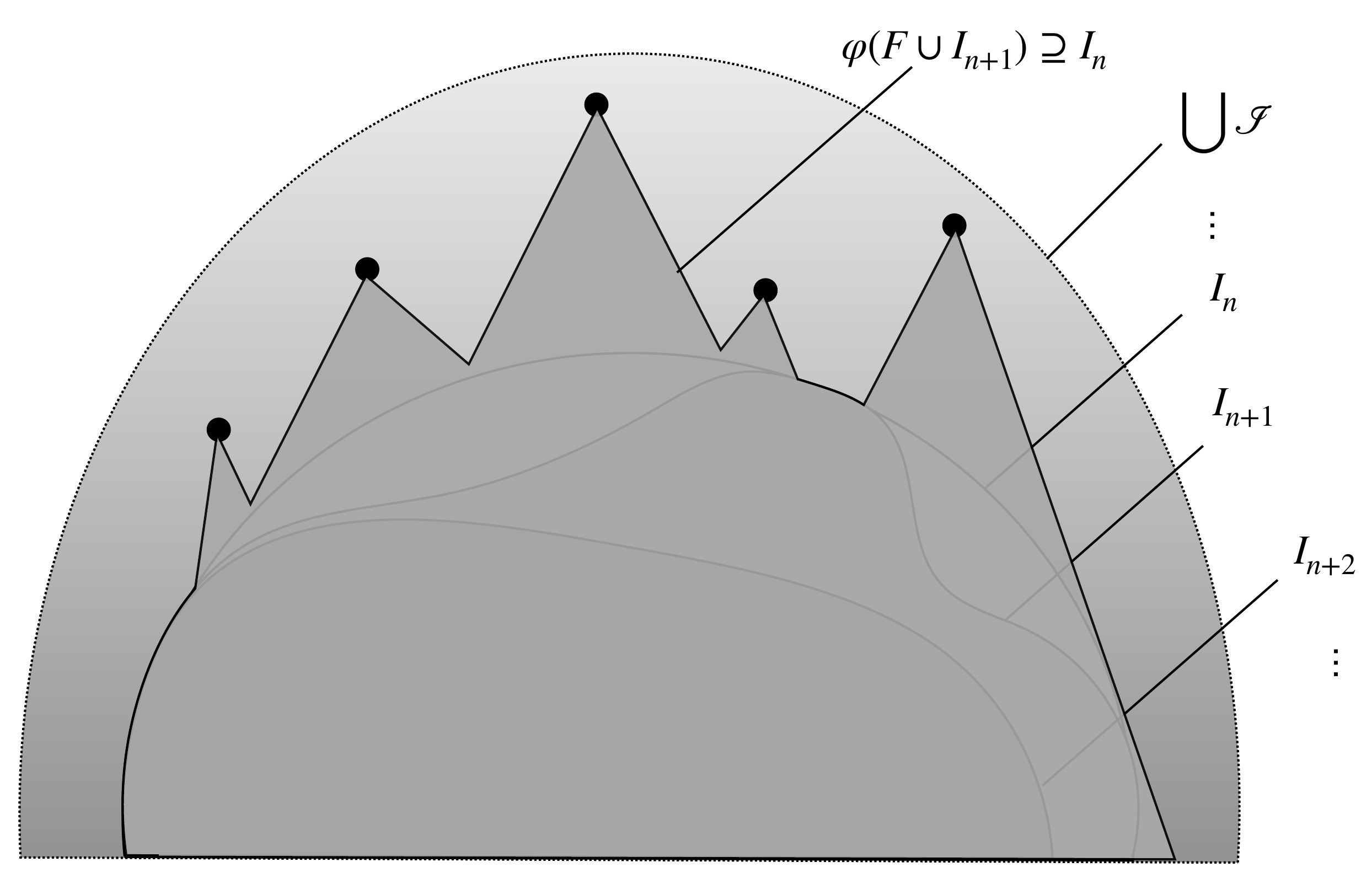}
    &
    \includegraphics[scale=0.15]{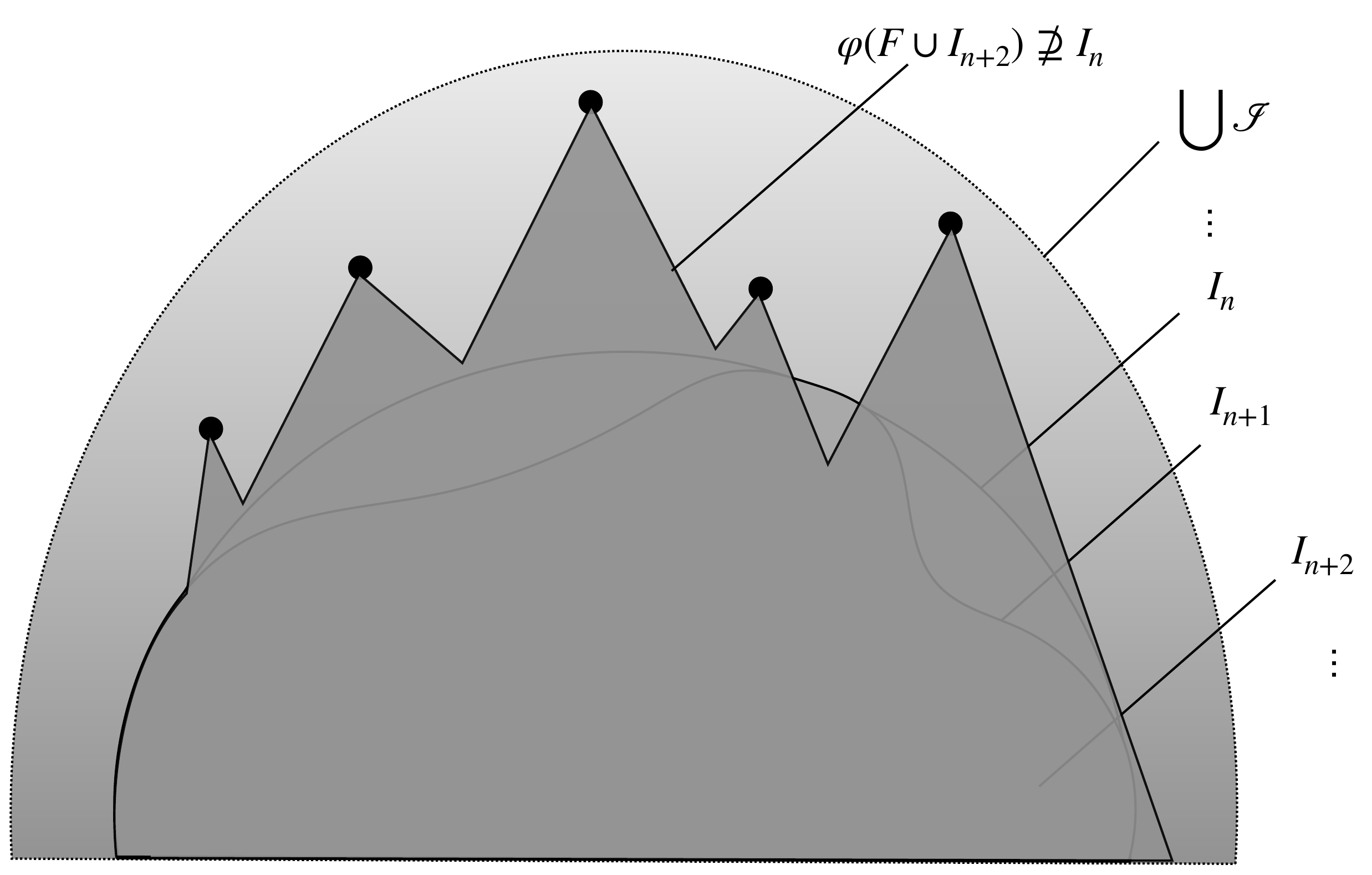}
  \end{tabular}
  \caption{A separating chain}
  \label{fig:chain}
\end{figure}

We will use the following notion and result, adapted from
\cite[Section 3, p.7]{chakir-pouzet}, see also \cite[p.25]{chakir}. A
non-empty 
chain $\mathcal I$ of closed sets of $(E, \varphi)$ is {\it
  separating} \index{separating} if for every
$I\in\mathcal I\setminus\{\bigcup\mathcal I\}$ and every finite set
$F \subseteq \bigcup\mathcal I\setminus I$, there is a set
$J\in\mathcal I$ such that $I\not\subseteq \varphi(F\cup J)$.  Note
that $J$ is necessarily included in $I$, and strictly so: otherwise,
since $\mathcal I$ is a chain, $I$ would be included in $J$, hence in
$\varphi (F \cup J)$.  In particular, a separating chain $\mathcal I$
cannot have a smallest element $I \neq \bigcup\mathcal I$; an
\emph{infinite} separating chain cannot have a smallest element at
all.  We illustrate the notion in Figure~\ref{fig:chain}.  Here we
assume that $\varphi$ is a topological closure operator, so that it
commutes with unions; $\mathcal I$ is a countable chain
$I_0 \supseteq I_1 \supseteq \cdots \supseteq I_n \supseteq \cdots$,
we consider $I := I_n$, and $F$ is the finite set of bullets in the
top tier of the figure.  Below left, we show
$\varphi (F \cup I_{n+1})$, which is the union of $I_{n+1}$ with the
closures of each of the points in $F$ (each closure of a point being
shown as a triangular region that extends below the point), and
$\varphi (F \cup I_{n+1})$ contains $I_n$; but
$\varphi (F \cup I_{n+2})$, shown below right in the figure, does not,
and therefore we can take $J := I_{n+2}$.

The following holds for every closure system, not necessarily
topological.

\begin{lemma}
  \label{independent}
  A closure system $(E, \varphi)$ contains an infinite independent set
  if and only if $E$ contains a subset $E'$ that contains an infinite
  separating chain of closed sets with respect to the induced closure.
\end{lemma}
\begin{proof}
  Any infinite independent set contains a countably infinite
  independent subset, so we may as well assume a given infinite
  independent set $X$ of the form $\{x_{n}: n<\omega\}$, where
  $x_m \neq x_n$ for all $m \neq n$. Set $E':=X$. Then the chain
  $\mathcal I=\{I_{n}: n<\omega\}$, where
  $I_{n}:=\varphi_{\restriction E'}(X\setminus\{x_{i}: i< n\})$, is
  separating in $E'$.  Indeed, first
  $\bigcup\mathcal I = I_0 = E' = X$, next every
  $I \in \mathcal I \setminus \{\bigcup\mathcal I\}$ is an $I_n$ for
  some $n \geq 1$. For every finite set $F$ of points of
  $I_0 \setminus I_n$, define $J$ as $I_{n+1}$.  Since $x_n$ is
  different from every $x_i$, $i < n$, $x_n$ is in
  $X\setminus\{x_{i}: i< n\}$ hence in $I_n$.  It follows that $x_n$
  is not in $F$.  It is not in $J$ either, because
  $J \subseteq \varphi (X \setminus \{x_i : i < n+1\}) \subseteq
  \varphi (X \setminus \{x_n\})$, and $X$ is independent.  Therefore
  $x_n$ is not in $F \cup J$.  We rewrite that as
  $F \cup J \subseteq X \setminus \{x_n\}$, and conclude that
  $\varphi (F \cup J) \subseteq \varphi (X \setminus \{x_n\})$.  Since
  $X$ is independent again, $x_n$ cannot be in $\varphi (F \cup J)$.
  However, $x_n$ is in $I_n$, so
  $I_n \not\subseteq \varphi (F \cup J)$.


  Conversely, let $E'$ be a subset of $E$ such that the induced
  closure $\varphi':= \varphi_{\restriction E'}$ on $E'$ contains an
  infinite separating chain $\mathcal I$ of closed sets. We construct
  an infinite independent subset for the induced closure $\varphi'$,
  and therefore also for the closure $\varphi$. To this end, we will
  define inductively an infinite sequence
  $x_{0}, I_{0}, \ldots, x_{n},I_{n}, \ldots$ such that
  $I_{0}\in\mathcal I\setminus\{\bigcup\mathcal I\}, x_{0}\in
  \bigcup\mathcal I \setminus I_{0}$ and such that, for every
  $n \geq 1$:
  \begin{enumerate}
  \item[$(a_n)$] $I_{n}\in\mathcal I$;
  \item[$(b_n)$] $I_{n}\subset I_{n-1}$;
  \item[$(c_n)$] $x_{n}\in I_{n-1}\setminus \varphi'(\{x_{0},  \ldots, 
    x_{n-1}\}\cup I_{n})$.
  \end{enumerate}
  Since $\mathcal I$ is infinite,
  $\mathcal I\setminus\{\bigcup\mathcal I\}\not =\emptyset$.  We
  choose $I_{0}\in\mathcal I\setminus\{\bigcup\mathcal I\}$ and
  $x_{0}\in\bigcup\mathcal I\setminus I_{0}$ arbitrarily.  Now let
  $n\geq 1$.  Let us assume that $x_{k}, I_{k}$ are defined and
  satisfy $(a_{k})$, $(b_{k})$, $(c_{k})$ for every $k\leq n-1$, and
  let us define $I:=I_{n-1}$ and $F:=\{x_{0}, \ldots,
  x_{n-1}\}$. Since $I\in\mathcal I$ and $F$ is a finite subset of
  $\bigcup\mathcal I\setminus I$, there is some $J\in\mathcal I$ such
  that $I\not\subseteq \varphi' (F\cup J)$.  The set $J$ is a proper
  subset of $I$: otherwise, since $\mathcal I$ is a chain, we would
  have $I \subseteq J$, hence
  $I \subseteq J \subseteq F \cup J \subseteq \varphi' (F \cup J)$.
  We pick $z\in I\setminus \varphi'(F \cup J)$, and we set $x_{n}:=z$,
  $I_{n}:=J$.
  
  It remains to check that the set $X:=\{x_{n}: n<\omega\}$ is
  independent. For every $x \in X$, say $x=x_n$, we know that $x_n$ is
  in $I_{n-1}$ and not in the closed set
  $C := \varphi'(\{x_{0}, \ldots, x_{n-1}\}\cup I_{n})$, by $(c_{n})$.
  The set $C$ contains $x_0$, \ldots, $x_{n-1}$. For every $k > n+1$,
  $I_{k-1} \subset I_n$ by $(b_{k-1})$, \ldots, $(b_{n+1})$, and
  $x_k \in I_{k-1}$ by $(c_k)$, so $I_n$ contains every $x_k$ with
  $k>n+1$.  It also contains $x_{n+1}$, and since $I_n \subseteq C$,
  $C$ contains every $x_k$ with $k \geq n+1$.  It follows that $C$
  contains every element of $X$ except $x=x_n$. Therefore
  $C \supseteq \varphi' (X \setminus \{x\})$, from which we conclude
  that $x$, not being in $C$, is not in $\varphi' (X \setminus \{x\})$
  either.
\end{proof}

\begin{remark}
  \label{remark1}
  Assuming $\varphi$ topological, we can take $E':= E$ in
  Lemma~\ref{independent}.  Indeed, let $X$ be an infinite independent
  subset of the form $\{x_{n}: n<\omega\}$, where the points $x_n$ are
  pairwise distinct. Let $I_{n}:=\varphi(X\setminus\{x_{i}: i<
  n\})$. The chain $\mathcal I=\{I_{n}: n<\omega\}$ is separating, as
  we now check.  First $\bigcup\mathcal I = I_0 = \varphi (X)$. Every
  $I \in \mathcal I \setminus \{\bigcup\mathcal I\}$ is an $I_n$ for
  some $n \geq 1$. For every finite set $F$ of points of
  $\varphi (X) \setminus I_n$, we define $J$ as $I_{n+1}$ and we check
  that $I_{n} \not\subseteq \varphi (F \cup J)$ by showing that $x_n$,
  which is in $I_n$, is not in $\varphi (F \cup I_{n+1})$.  Since
  $\varphi$ is topological, it suffices to show that $x_n$ is neither
  in $\varphi (F)$ nor in $\varphi (I_{n+1}) = I_{n+1}$.  The
  latter---that $x_n$ is not in $I_{n+1}$---is clear.  As for the
  former, we note that
  $I_n \cup \varphi (X \setminus \{x_n\}) = \varphi ((X \setminus
  \{x_i : i < n\}) \cup (X \setminus \{x_n\})) = \varphi (X)$, using
  the fact that $\varphi$ is topological.  That implies
  $\varphi (X) \setminus I_n \subseteq \varphi (X \setminus \{x_n\})$,
  so $F \subseteq \varphi (X \setminus \{x_n\})$, whence
  $\varphi (F) \subseteq \varphi (X \setminus \{x_n\})$.  If $x_n$
  were in $\varphi (F)$, it would then be in
  $\varphi (X \setminus \{x_n\})$, and that is impossible since $X$ is
  independent.
\end{remark}

\begin{remark}
  \label{remark2}
  In general, namely if $\varphi$ is not topological, we cannot take
  $E:=E'$.

  Let us consider $E := \pow (\N)$, and let us define $\varphi$ by
  $\varphi (\mathcal A) := \pow (\bigcup \mathcal A)$ for every
  $\mathcal A \in \pow (E)$.  The closed subsets of $E$ are exactly
  the sets of the form $\pow (A)$ with $A \subseteq \N$.  There is an
  infinite independent set, say $X := \{\{n\} \mid n \in \N\}$.
  However, no subset of $E$ has any separating chain of closed sets
  $\mathcal I$.  Indeed, let us assume that one existed, and let us
  pick any $I \in \mathcal I \setminus \{\bigcup \mathcal I\}$.  Since
  $I$ is closed, $I = \pow (A)$ for some $A \subseteq \N$.  We take
  $F := \{A\}$: for every $J \in \mathcal I$,
  $\varphi (F \cup J) \supseteq \varphi (F) = \pow (A) = I$, so
  $\mathcal I$ cannot be separating.
\end{remark}

In a quasi-ordered set $E$, for every $A \subseteq E$, we write
$\dc A$ for $\{x \in E : \exists y \in A, x \leq y\}$. A subset $A$ is
\emph{cofinal} in $E$ if and only if $E = \dc A$ \cite[Chapter~2, 5.1,
p.44]{fraisse}. An \emph{initial segment} is a subset $I$ of $E$ such
that $I = \dc I$. The \emph{finitely generated} initial segments are
the sets of the form $\dc A$, $A$ finite. We write $\dc x$ for
$\dc \{x\}$.

Every closure system $E$ is quasi-ordered by $x\leq y$ if and only if
$x\in \varphi(\{y\})$. For every $x \in X$, $\dc x = \varphi (\{x\})$.
Every closed set is an initial segment. If $\varphi$ is topological,
then every finitely generated initial segment $\dc A$ is equal to
$\varphi (A)$, hence is closed.


\ifcard
\begin{lemma}
  \label{cofinality1}
  Let $(E, \varphi)$ be a closure system and $\kappa$ be the dimension
  of $(E, \varphi)$. Then $E$ contains a generating subset $D$ of
  cardinality $\kappa$ on which the collection of finitely generated
  initial segments of the quasi-ordered set
  $(D, \leq_{\restriction D})$ is well-founded.
\end{lemma}

According to a result of Birkhoff (1937), the poset of finitely
generated initial segments of a poset is well-founded whenever the
poset is well-founded \cite[Theorem2, p.182]{birkhoff}. This property
holds for initial segments of a quasi-ordered set too, provided we
define a quasi-ordered set as \emph{well-founded} if it has no
infinite strictly descending sequence
$x_0 > x_1 > \cdots > x_n > \cdots$, where $x < y$ if and only if
$x \leq y$ and $y \not\leq x$. Thus, in order to prove
Lemma~\ref{cofinality1} it suffices to prove that $E$ contains a
generating subset on which the quasi-order $\leq$ is well-founded.
This fact follows readily from the following lemma, which was already
known to hold for the closure of initial segments of a poset
\cite{milner-pouzet82}.

\begin{lemma}
  \label{cofinality2}
  Given a closure system $(E, \varphi)$ of dimension $\kappa$, there
  is an enumeration $(x_{\alpha})_{\alpha<\kappa}$ of a generating
  subset $D$ of $E$ of cardinality $\kappa$ such that
  $x_{\beta} \not \leq x_{\alpha}$ for all $\alpha <\beta < \kappa$.
\end{lemma}
\begin{proof} Let $E'$ be a generating subset of $E$ of cardinality
  $\kappa$. Let $(y_{\alpha})_{\alpha<\kappa}$ be an enumeration of
  $E'$. Define a sequence $(\theta (\alpha))_{\alpha < \kappa}$ as
  follows. Let $\beta<\kappa$ and suppose $\theta(\alpha)$ is defined
  for every $\alpha <\beta$. Since $\beta < \kappa$,
  $\vert \{y_{\theta(\alpha)}: {\alpha<\beta}\}\vert < \kappa$, hence
  $\{y_{\theta(\alpha)}: \alpha<\beta\}$ is not a generating set.  In particular, some
  $y_{\gamma}$ is not in $\varphi (\{y_{\theta(\alpha)}:
  {\alpha<\beta}\})$. Let $\theta (\beta)$ be the least $\gamma$ with
  this property. For example, $\theta(0)=0$. Set
  $x_{\alpha}:= y_{\theta (\alpha)}$ for every $\alpha<\kappa$.

  For all $\alpha < \beta < \kappa$, if $x_\beta \leq x_\alpha$ then
  we would have $y_{\theta (\beta)} \leq y_{\theta (\alpha)}$, namely
  $y_{\theta(\beta)} \in \varphi (\{y_{\theta (\alpha)}\})$, which is
  impossible since $y_{\theta (\beta)}$ is not in the larger set
  $\varphi (\{y_{\theta(\alpha)}:
  {\alpha<\beta}\})$. Hence $x_\beta \not\leq x_\alpha$.

  Furthermore, the set $D:= \{x_{\alpha}: \alpha<\kappa\}$ is a
  generating set. For that, it suffices to observe that
  $E' \subseteq \varphi (D)$, since that inclusion implies
  $E= \varphi(E') \subseteq \varphi(\varphi (D))= \varphi(D)$. If
  $E' \not \subseteq \varphi (D)$, then some $y_\gamma$ is not in
  $\varphi (D)$, $\gamma < \kappa$.  For every $\beta < \kappa$, since
  $y_\gamma$ is not in $\varphi (D)$, it is not in the smaller set
  $\varphi (\{y_{\theta(\alpha)}:{\alpha<\beta}\})$. By the way we
  have constructed $\theta (\beta)$ as a minimal ordinal,
  $\theta (\beta) \leq \gamma$. This holds for every $\beta < \kappa$.
  Since the ordinals $\theta (\beta)$, $\beta < \kappa$, are pairwise
  distinct, there are at least $\kappa$ distinct ordinals below
  $\gamma$. That is impossible since $\gamma < \kappa$.
\end{proof}
\else

We say that a quasi-ordered set is \emph{well-founded} if and only if
it has no infinite strictly descending sequence
$x_0 > x_1 > \cdots > x_n > \cdots$, where $x < y$ if and only if
$x \leq y$ and $y \not\leq x$.  That extends the same notion on
posets.

We will use the following lemma in order to prove the implication
$(i) \Rightarrow (iii)$.

\begin{lemma}
  \label{cofinality1}
  If $(E, \varphi)$ is a closure system then $E$ contains a generating
  subset $D$ such that the collection of finitely generated initial
  segments $\mathbf I_{<\omega}(D)$ of the quasi-ordered set
  $(D, \leq_{\restriction D})$ is well-founded under inclusion.
\end{lemma}
\begin{proof}
  According to a result of Birkhoff, 
  the poset $\mathbf I_{<\omega}(P)$ of finitely generated initial
  segments of a poset $P$ is well-founded if $P$ is well-founded
  \cite[Theorem~2, p.182]{birkhoff}. This property holds for initial
  segments of a quasi-ordered set too, since initial segments of a
  quasi-ordered set are inverse images of initial segments of the
  order quotient. For the reader's convenience, we give a proof of
  Birkhoff's result. Let
  $\dc A_1 \supset \dc A_2 \supset \cdots \supset \dc A_n \supset$ be
  an infinite descending sequence where each $A_i$ is finite. We may
  assume that each $A_i$ is an antichain. We construct a tree $T$
  whose vertices are finite chains $\{x_1, x_2, \cdots, x_n\}$ where
  each $x_i \in A_i$ and $x_1 \geq x_2 \geq \cdots \geq x_n$. The
  cardinality of such a set is $n$, or some lower number (if some
  element appears several times), and is the \emph{depth} of the
  vertex in $T$. The unique parent of a depth $n$ set, $n \geq 1$, is
  obtained by removing its least element ($x_n$, but also $x_{n-1}$ if
  that happens to be equal to $x_n$, and so on). The empty chain is
  the root. Every element of each $A_n$ is the least element of at
  least one such chain, of depth at most $n$. Hence every set
  $\dc A_n$ appears as the initial segment generated by a finite set
  of vertices of $T$. Since there are infinitely many sets $\dc A_n$,
  the tree $T$ is infinite. Since $T$ is finitely branching, by K\H
  onig's Lemma \cite{Konig:lemma} it has an infinite branch, and that
  is an infinite descending sequence of elements: contradiction.

  Thus, in order to prove the lemma, it suffices to prove that $E$
  contains a generating subset on which the quasi-order $\leq$ is well-founded.

  A result due to Hausdorff (see \cite[Chapter~2, p.57]{fraisse})
  states that every poset contains a well-founded cofinal subset. That
  is also valid for quasi-ordered sets such as $E$.  Indeed, let us
  consider the family $\mathcal W$ of all well-founded subsets of the
  set $E$, and let us order it by prefix: $A \sqsubseteq B$ if and
  only if $B \cap \dc A = A$.
  By Zorn's Lemma, it has a maximal element $A$. If $A$ were not
  cofinal, there would be a point $x$ that is not in $\dc A$. Then
  $B := A \cup \{x\}$ would be a strictly larger well-founded subset of
  $E$, contradicting maximality.

  Hence let $D$ be a well-founded cofinal subset of $E$. For every
  $x \in E$, there is a point $y \in D$ such that $x \leq y$. In other
  words, $x$ is in $\varphi (\{y\})$, hence in the larger set
  $\varphi (D)$. Therefore $\varphi (D) = E$ and $D$ is generating.
\end{proof}

\subsection*{Proof of the implication $(i)\Rightarrow (iii)$.}
\label{sec:proof-1}

%
Let $\varphi$ be the closure associated with the collection of closed
sets $\mathcal F$, and let us remember that it is topological. Let $C$
be a closed set. We define a dense subset of $C$ on which the induced
topology is Noetherian.
This will be $D$, as given by Lemma \ref{cofinality1}.  We note that $D$ is
dense in $C$, begin a generating set.

Let $\varphi_{\restriction D}$ be the closure induced on $D$, namely
$\varphi_{\restriction D}(X) := \varphi(X)\cap D$ for every
$X \subseteq D$. This is also a topological closure, and we claim that
it is Noetherian.

By $(i)$, $E$ contains no infinite discrete set, so $D$ does not
contain any infinite discrete set either.  Let us imagine that $D$
contained an infinite strictly descending sequence
$I_0 \supset I_1 \supset \cdots \supset I_n \supset \cdots$ of closed
subsets. By Lemma~\ref{independent},
that chain must fail to be separating: there must be an index
$m_1 \geq 1$ and a finite set $F_1$ of points of
$I_0 \setminus I_{m_1}$ such that, for every $n < \omega$ (in
particular, for every $n > m_1$),
$I_{m_1} \subseteq \varphi_{\restriction D} (F_1 \cup I_n) =
\varphi_{\restriction D} (F_1) \cup I_n$. The last equality is because
$\varphi_{\restriction D}$ is topological, and $I_n$ is closed. Hence
$m_1 \geq 1$, $F_1 \subseteq I_0 \setminus I_{m_1}$, and for every
$n > m_1$,
$I_{m_1} \setminus I_n \subseteq \varphi_{\restriction D} (F_1)$. We
do the same with the infinite subsequence starting at $I_{m_1}$: there
is an index $m_2 \geq m_1+1$ and a finite set
$F_2 \subseteq I_{m_1} \setminus I_{m_2}$ such that for every
$n > m_2$,
$I_{m_2} \setminus I_n \subseteq \varphi_{\restriction D} (F_2)$.
Proceeding this way, we obtain indices $m_{k+1} \geq m_k+1$ and finite
sets $F_{k+1} \subseteq I_{m_k} \setminus I_{m_{k+1}}$ such that for
every $n > m_{k+1}$,
$I_{m_{k+1}} \setminus I_n \subseteq \varphi_{\restriction D}
(F_{k+1})$, for every $k \geq 1$.

Since
$F_{k+1} \subseteq I_{m_k} \setminus I_{m_{k+1}} \subseteq
\varphi_{\restriction D} (F_k)$, we have
$\varphi_{\restriction D} (F_{k+1}) \subseteq \varphi_{\restriction D}
(F_k)$, for every $k \geq 1$. It follows that the sequence
${(\varphi_{\restriction D} (F_k))}_{k \geq 1}$ is descending. Since
$D$ was obtained from Lemma \ref{cofinality1}, that sequence must be
finite.  Let us pick $k \geq 2$ such that
$\varphi_{\restriction D} (F_k) = \varphi_{\restriction D} (F_{k+1})$.
The set $F_k$ cannot be empty, since $\varphi_{\restriction D} (F_k)$
contains $I_{m_k} \setminus I_{m_{k+1}}$, which is non-empty.  We pick
$x \in F_k$.  In particular, $x$ is in
$I_{m_{k-1}} \setminus I_{m_k}$, hence is not in $I_{m_k}$. However,
$x$ is also in
$\varphi_{\restriction D} (F_k) = \varphi_{\restriction D} (F_{k+1})$,
and since
$F_{k+1} \subseteq I_{m_k} \setminus I_{m_{k+1}} \subseteq I_{m_k}$,
$x$ is also in $\varphi_{\restriction D} (I_{m_k}) = I_{m_k}$:
contradiction.  \qed

\subsection*{Proof of the implication $(iii) \Rightarrow (ii)$.}

Let $\varphi$ be the closure on $E$. Let $C$ be a closed set and $D$
be a dense subset of $C$ on which the closure
$\varphi_{\restriction D}$ is well-founded.

On $D$ every closed set $D'$ is a finite union of irreducible closed
sets. This fact goes back to Noether, see \cite[Chapter VIII,
Corollary, p.181]{birkhoff}. Indeed, if $D$ is not such, then, since
the collection of closed sets on $D$ is well-founded, there is a
minimal member $C'$ which is not a finite union of irreducible
members. In particular, $C'$ is non-empty. If $C'$ is the union of two
proper closed subsets, by minimality those closed subsets must be
finite unions of irreducible subsets of $D$, hence so must be $C'$. It
follows that $C'$ is irreducible: contradiction.

Now $D$ is itself closed in $D$, so we can write $D$ as a finite union
of irreducible closed subsets $C_i$ of $D$, $1\leq i\leq n$. For each
$C_i$, $\varphi (C_i)$ is irreducible in $E$, as one easily checks
\cite[Lemma~8.4.10]{goubault-larrecq1}. By density and the fact that
$\varphi$ is topological,
$C = \varphi (D) = \bigcup_{i=1}^n \varphi (C_i)$.  \qed

\subsection*{Proof of the implication $(ii) \Rightarrow (i)$}
\label{sec:proof-ii-rightarrow}

Let $\varphi$ be the closure on $E$ again, and let $X$ be a discrete
subspace. We write $\varphi (X)$ as a finite union of irreducible
closed sets $I_1$, \ldots, $I_n$.

For each $x \in X$, $x$ is in some $I_k$. We claim that
$I_k = \varphi (\{x\})$. To this end, we note that
$I_k \subseteq \varphi (X)= \varphi (\{x\}) \cup \varphi (X\setminus
\{x\})$, since $\varphi$ is topological. Therefore $I_k$ is equal to
the union of the two closed sets $\varphi (\{x\}) \cap I_k$ and
$\varphi (X\setminus \{x\}) \cap I_k$. Since $X$ is discrete, hence
independent, $x$ is not in $\varphi (X\setminus \{x\})$, and since
$x \in I_k$, $\varphi (X\setminus \{x\}) \cap I_k$ is a proper closed
subset of $I_k$. Because $I_k$ is irreducible,
$\varphi (\{x\}) \cap I_k$ cannot be a proper closed subset of $I_k$,
so $\varphi (\{x\}) \cap I_k = I_k$. This means that
$I_k \subseteq \varphi (\{x\})$, and the converse inclusion follows
from $x \in I_k$.

It follows that for any two distinct points $x, y \in X$, $x$ and $y$
cannot be in the same $I_k$. Otherwise
$\varphi (\{x\}) = \varphi (\{y\})$, but since $X$ is independent, $x$
is not in $\varphi (X \setminus \{x\})$, hence not in the smaller set
$\varphi (\{y\})$. That is impossible since
$\varphi (\{x\}) = \varphi (\{y\})$ contains $x$.

Since each $I_k$ can contain at most one point from $X$, $X$ is
finite.  \qed

\section{Remarks and comments}
\label{sec:remarks}

\subsection{Other characterizations}
\label{sec:other-char}

A.H. Stone \cite[Theorem~2]{stone} shows that $(i)$ is equivalent to
two further properties: $(iv)$ every open cover of every subspace $X$
of $T$ has a finite subfamily whose union is dense in $X$, and $(v)$
every continuous real-valued function on every subspace of $T$ is
bounded.

He also shows \cite[Theorem~3]{stone} that, for every topological
space $X$, the fact that $X$ is a finite union of irreducible closed
subsets---the special case of $(ii)$ where the closed set is the whole
of $X$---is equivalent to six other conditions, among which: $(a)$
$X$ is \emph{semi-irreducible}, namely every collection of pairwise
disjoint non-empty open sets is finite, and: $(b)$ $X$ has only
finitely many regular open sets.  He also observes that the
cardinality of such collections must be bounded \cite[Theorem~3,
$(vi)$]{stone}.

In particular, our condition $(ii)$ is equivalent to: $(vi)$ every
closed subspace of $T$ is semi-irreducible, and to: $(vii)$ every
closed subspace of $T$ has only finitely many regular open sets.

We will come back to this in Section~\ref{sec:posets}.

\subsection{Noetherian topological spaces}
\label{sec:noeth-topol-spac}

Noetherian topological spaces have been studied for their own sake by
A.H. Stone \cite{stone}. They are an important basic notion in
algebraic geometry, since the spectrum of any Noetherian ring in a
Noetherian topological space, with the Zariski topology.  They have
also found applications in verification, the domain of computer
science concerned with finding algorithms that prove properties of
other computer systems, automatically \cite{JGL-icalp10}. One can
consult Section~9.7 of \cite{goubault-larrecq1}, which is devoted to
Noetherian topological spaces.


\subsection{A related result on general closure systems}
\label{sec:related-result}

The implication $(i)\Rightarrow (ii)$ follows from the following
result about closure systems, as we will see. We recall that an
\emph{up-directed} subset of a poset $P$ is a non-empty subset $A$ of
$P$ such that any two elements of $A$ have an upper bound in $A$, and
that an \emph{ideal} is an up-directed initial segment. We always
order powersets by inclusion.

\begin{theorem}
  \label{thm:galvin-milner-pouzet}
  \cite[Theorem~1.2]{galvin-milner-pouzet} If a closure system
  $(E, \varphi)$ contains no infinite independent set then: $(*)$
  there are finitely many pairwise disjoint subsets $A_i$ ($i \in I$)
  of $E$ and, for each $A_i$, a proper ideal $N_i$ of $\pow (A_i)$
  such that for every $X\subseteq \bigcup_{i \in I} A_i$, the set $X$
  generates $E$ if and only if $A_i \cap X\not \in N_i$, for each
  $i \in I$.
\end{theorem}

As it will become apparent in Proposition \ref{prop:topo}, this result
specialized to topological closures is just implication
$(i)\Rightarrow (ii)$. Decompositions of topological closures were
considered in \cite{milner-pouzet82}, but this consequence was totally
missed.

\begin{remark}
  \label{rem:thgmp}
  Condition $(*)$ entails that $\bigcup_{i \in I} A_i$ generates $E$.
\end{remark}

\begin{remark}
  \label{rem:theorgmp}
  In Theorem~\ref{thm:galvin-milner-pouzet}, and if $\varphi$ is
  topological, we may suppose that:
  \begin{equation}\label{equ:key}
    A_i\cap \varphi (\bigcup_{j\in I\setminus\{i\}} A_j)=\emptyset \; \text{for each}\;  i\in I. 
  \end{equation}
\end{remark}
This is Remark~1 of \cite{milner-pouzet82}.  We repeat the argument.
Let $A_i$ and $N_i$ satisfy condition $(*)$ of
Theorem~\ref{thm:galvin-milner-pouzet}.  If $(\ref{equ:key})$ does not
hold, we set
$A'_i:= A_i\setminus \varphi (\bigcup_{j\in I\setminus\{i\}} A_i)$ and
$N'_i:=N_i\cap \pow (A'_i)$.  First, we claim that each $N'_i$ is a
proper ideal of $A'_i$.  To this end, we set
$X:= A'_i \cup \bigcup_{j\not = i} A_j$.  Because $\varphi$ is
topological,
$\varphi (X) = \varphi (A'_i) \cup \varphi (\bigcup_{j\not = i} A_j)
\supseteq A'_i \cup \varphi (\bigcup_{j\not = i} A_j) \supseteq A_i$;
also, $\varphi (X) \supseteq X \supseteq A_j$ for every $j \neq i$, so
$\varphi (X) \supseteq \bigcup_{j \in I} A_j$.  Remark~\ref{rem:thgmp}
then entails that $X$ generates $E$.
By $(*)$, $A'_i \cap X = A'_i$ cannot be in $N_i$;
therefore $A'_i$ is not in $N'_i$, showing that $N'_i$ is a proper
ideal of $\pow (A'_i)$.

Next, let $X\subseteq \bigcup_{i \in I} A'_i$.  Let us assume that $X$
generates $E$.  Since $X\subseteq \bigcup_{i \in I} A_i$, we may use
$(*)$: $A_i \cap X\not \in N_i$, for each $i \in I$. Hence
$A'_i \cap X= A_i\cap X\not \in N'_i$.  Conversely, if
$A'_i \cap X\not \in N'_i$ for each $i \in I$, then
$A'_i \cap X \not\in N_i$, so $A_i \cap X = A'_i \cap X \not\in N'_i$,
hence $X$ generates $E$ by $(*)$.

\begin{proposition}
  \label{prop:topo}
  Let $\varphi$ be a topological closure operator on a set $E$. Then
  $E$ is a finite union of irreducible closed sets iff $E$ has a
  decomposition satisfying condition $(*)$ on generating sets of
  Theorem \ref{thm:galvin-milner-pouzet}.
\end{proposition}
\begin{proof}
  The result is a consequence of the following two claims.  Each one
  establishes one direction of the implication.
  \begin{claim}
    \label{claim1}
    Let $A_i$, $N_i$ ($i\in I$) be a finite decomposition satisfying
    condition~$(*)$ on generating sets of
    Theorem~\ref{thm:galvin-milner-pouzet}. According to Remark
    \ref{rem:theorgmp}, we may assume that it satisfies Condition
    (\ref{equ:key}). Then, for every subset $Y$ of $E$, $Y\in N_i$ iff
    $Y\subseteq A_i$ and $\varphi(Y)\not \supseteq A_i$. In
    particular, $X_i:= \varphi(A_i)$ is irreducible.
  \end{claim}
  \begin{proof}[Proof of Claim~\ref{claim1}.]
    By Remark~\ref{rem:thgmp}, $E = \varphi (\bigcup_{j\in I} A_j)$.
    Let us assume that $Y\in N_i$. Since $N_i$ is an initial segment,
    $A_i \cap Y$ is in $N_i$.  For
    $X:=Y\cup \bigcup_{j\in I\setminus \{i\}} A_j$,
    $A_i \cap X = A_i \cap Y$ since the sets $A_j$ are pairwise
    disjoint, so $A_i \cap X$ is in $N_i$.  By $(*)$, $X$ does not
    generate $E$.  If $A_i\subseteq \varphi (Y)$ then
    $E=\varphi (\bigcup_{j\in I} A_j)\subseteq \varphi (\varphi
    (Y)\cup \bigcup_{j\in I\setminus \{i\}} A_j) = \varphi (Y \cup
    \bigcup_{j\in I\setminus \{i\}} A_j)$ (since $\varphi$ is a
    topological closure operator) $= \varphi (X)$, which is
    impossible. Hence $A_i\not \subseteq \varphi(Y)$.

    Conversely, if $\varphi(Y)\not \supseteq A_i$, then there is a
    point $x$ in $A_i$---hence not in
    $\varphi (\bigcup_{j\in I\setminus\{i\}} A_i)$ by
    Condition~(\ref{equ:key})---which is not in $\varphi (Y)$, hence
    not in
    $\varphi(Y) \cup \varphi (\bigcup_{j\in I\setminus \{i\}} A_j)$.
    The latter is equal to $\varphi(X)$, where
    $X := Y\cup \bigcup_{j\in I\setminus \{i\}} A_j$, since $\varphi$
    is topological, so $X$ does not generate $E$. Using $(*)$,
    $A_j \cap X$ is in $N_j$ for some $j \in I$. If $j \neq i$, then
    $A_j \cap X \supseteq A_j$ by definition of $X$, and
    $A_j \cap X \in N_j$ implies $A_j \in N_j$, contradicting the fact
    that $N_j$ is proper.  Therefore $j=i$. This means that
    $A_i \cap X$, which is equal to $A_i \cap Y$ since the sets $A_j$
    are pairwise disjoint, hence to $Y$ since $Y \subseteq A_i$, is in
    $N_i$.

    We finally show that $X_i = \varphi (A_i)$ is irreducible. Since
    $N_i$ is proper, $A_i$ is non-empty, hence $X_i$ is non-empty. Let
    us assume that $X_i$ is the union of two proper closed subsets
    $C_1$ and $C_2$.  We consider $Y:=C_1 \cap A_i$ (resp.,
    $Y:=C_2 \cap A_i$). Then $Y \subseteq A_i$ and
    $\varphi (Y) \subseteq C_1$ (resp., $C_2$); in particular,
    $\varphi (Y)$ cannot contain $X_i = \varphi (A_i)$, hence cannot
    contain $A_i$.  By the first part of the claim, $Y$ is in
    $N_i$.  In other words, both $C_1 \cap A_i$ and $C_2 \cap A_i$ are
    in $N_i$. Since $N_i$ is an ideal,
    $(C_1 \cap A_i) \cup (C_2 \cap A_i) = (C_1 \cup C_2) \cap A_i =
    X_i \cap A_i = A_i$ is in $N_i$, which is impossible since $N_i$
    is proper.
  \end{proof}

  \begin{claim}
    \label{claim2}
    If $E$ is a finite union of irreducible closed sets, let
    $(X_i)_{i \in I}$
    be a family of such sets with $\vert I\vert$ minimum. Set
    $A_i:= X_i\setminus \bigcup _{j\not= i} X_j$ and
    $N_i:= \{A'\subseteq A_i: \varphi (A' )\not = X_i\}$. This
    decomposition satisfies Condition~$(*)$ of
    Theorem~\ref{thm:galvin-milner-pouzet}.
  \end{claim}
  \begin{proof}[Proof of Claim~\ref{claim2}.]
    We check that $N_i$ is an ideal. Given $A', B' \in N_i$,
    $\varphi (A' \cup B') = \varphi (A') \cup \varphi (B')$, since
    $\varphi$ is topological. If that were equal to the whole of
    $X_i$, and since $\varphi (A')$ and $\varphi (B')$ are both proper
    closed subsets of $X_i$, $X_i$ would fail to be irreducible. Hence
    $\varphi (A' \cup B') \neq X_i$, so that $A' \cup B'$ is in $N_i$.

    Then we check that $N_i$ is proper, namely that $A_i$ is not in
    $N_i$. By definition of $A_i$,
    $X_i \subseteq A_i \cup (\bigcup _{j\not= i} X_j)
    \subseteq 
    \varphi (A_i) \cup (\bigcup _{j\not= i} X_j)$, so
    $X_i = (X_i \cap \varphi (A_i)) \cup (X_i\cap (\bigcup _{j\not= i}
    X_j)) = \varphi (A_i) \cup (X_i\cap (\bigcup _{j\not= i} X_j))$, a
    union of two closed subsets. The second one,
    $X_i\cap (\bigcup _{j\not= i} X_j)$, is a proper subset of $X_i$
    since we have chosen a family of least cardinality. Since $X_i$
    is irreducible, the other one cannot be a proper subset.  Therefore
    $\varphi(A_i)=X_i$.  It follows that $A_i$ is not in $N_i$.

    Finally, let $X \subseteq \bigcup_{i\in I} A_i$.

    If $A_i\cap X$ belongs to $N_i$ for no $i \in I$, then by
    definition of $N_i$, $\varphi(A_i\cap X)= X_i$, hence
    $\varphi (X) = \varphi (\bigcup_{i\in I} A_i \cap X) = \bigcup_{i
      \in I} \varphi (A_i \cap X) = \bigcup_{i\in I} X_i = E$.

    Conversely, let us assume that $\varphi (X) = E$. Since
    $X \subseteq \bigcup_{j\in I} A_j$,
    $X = \bigcup_{j \in I} A_j \cap X$, and since $\varphi$ is
    topological,
    $E = \varphi (X) = \bigcup_{j \in I} \varphi (A_j \cap X)$.  We
    recall that $E = \bigcup_{i \in I} X_i$, so for every $i \in I$,
    $X_i \subseteq \bigcup_{j \in I} \varphi (A_j \cap X)$, and since
    $X_i$ is irreducible, there is a $j \in I$ such that
    $X_i \subseteq \varphi (A_j \cap X)$. If $j \neq i$, then
    $X_i \subseteq \varphi (A_j) \subseteq \varphi (X_j) = X_j$, which
    is impossible since we have chosen ${(X_i)}_{i \in I}$ of least
    cardinality. Hence $j=i$, meaning that for every $i \in I$,
    $X_i \subseteq \varphi (A_i \cap X)$. Since
    $\varphi (A_i \cap X) \subseteq \varphi (A_i) \subseteq \varphi
    (X_i) = X_i$, $X_i = \varphi (A_i \cap X)$, and that shows that
    $A_i \cap X$ is not in $N_i$.
  \end{proof}
  \let\qed\relax 
\end{proof}

\subsection{Topological properties versus lattice properties}

Item $(ii)$ in Theorem \ref{theorem-main} is a property about the
lattice of closed sets of a topological space: if two topological
spaces have isomorphic lattices of closed sets, then they both satisfy
$(ii)$ or neither one satisfies it.  A property of a space that only
depends on the isomorphism class of its lattice of closed sets is
called a \emph{lattice property}.  Hence $(ii)$ is a lattice property.

Item $(i)$, too, is a lattice property.  Indeed, as it is well-known,
the existence of an infinite discrete subspace (or more generally of
an infinite independent subset for a closure system) amounts to the
existence of an embedding of $\pow (\N)$, the collection of subsets of
$\N$ ordered by inclusion, into the lattice of closed sets.

It not clear that item $(iii)$ is a lattice property without having a
proof of Theorem \ref{theorem-main}.  In order to see why, let us
consider the \emph{sobrification} $X^s$ of a topological space $X$,
see Section~8.2.3 of \cite{goubault-larrecq1} for example.  This is a
construction that naturally occurs through the contravariant duality
between topological spaces and frames.  To say it briefly, $X^s$ is
the free sober space over $X$
\cite[Theorem~8.2.44]{goubault-larrecq1}.  A \emph{sober space} is a
$T_0$ space in which the irreducible closed subsets are the closures
of single points.  The sobrification $X^s$ can be obtained as the
collection of irreducible closed subsets of $X$, with the topology
whose open sets are (exactly) the sets $\diamond U$ defined as
$\{C \text{ irreducible closed } \mid C \cap U \neq \emptyset\}$,
where $U$ ranges over the open subsets of $X$.  The map
$U \mapsto \diamond U$ is then an order-isomorphism
\cite[Lemma~8.2.26]{goubault-larrecq1}.  It follows that a lattice
property cannot distinguish $X$ from its sobrification $X^s$.  And
density, as used in the statement of item~(iii), is not a lattice
property: any dense subset of $\N$ with the cofinite topology (whose
closed sets are the finite subsets of $\N$ plus $\N$ itself) must be
infinite, but $\N^s$, which is homeomorphic to the space obtained by
adding a new point $\infty$ to $\N$, and whose closed sets are the
finite subsets of $\N$ plus $\N^s$, has a one-point dense subset,
$\{\infty\}$, but an isomorphic lattice of closed sets.

\subsection{The case of posets}
\label{sec:posets}

Theorem \ref{theorem-main} has a well-known predecessor in the theory
of posets. It is worth to recall it.

Let $P$ be a poset, and $A$ be a subset of $P$. An \emph{upper bound}
of $A$ is any $z \in P$ such that $x \leq z$ for every $x \in A$.  Two
elements are \emph{compatible} if they have a common upper bound, and
\emph{incompatible} otherwise.  The set $A$ is \emph{up-independent} if
all its elements are pairwise incompatible; it is \emph{consistent} if
all its elements are pairwise compatible.

The \emph{final segments} of $P$ are the initial segments of $P^d$,
the opposite order; we denote by $\upc A$, resp.\ $\upc a$, the final
segment of $P$ generated by $A\subseteq P$, resp.\ $a\in P$.

The set $\mathbf I(P)$ of initial segments of $P$ is the set of closed
sets of a topology, the Alexandroff topology. In this setting, a
subset $A$ is discrete if and only if it is an antichain, $A$ is dense
if and only if it is cofinal, and $A$ is irreducible if and only if it
is an ideal.
 
A poset $P$
is \emph{well-quasi-ordered} (w.q.o.\ for short) if it is well-founded
and contains no infinite antichain. According to Higman \cite{higman},
$P$ is w.q.o. iff $\mathbf I(P)$ is well-founded.

We recall the following  result (see \cite[Chapter~4]{fraisse}):
\begin{theorem}
  \label{thm:fraisse}
  The following properties are equivalent for a poset $P$:
  \begin{enumerate}[label=(\alph*)]
  \item $P$ contains no infinite antichain;
  \item every initial segment of $P$ is a finite union of ideals;
  \item every initial segment of $P$ contains a cofinal subset which
    is well-quasi-ordered.
  \end{enumerate}
\end{theorem}
\begin{proof}
  This is just Theorem~\ref{theorem-main} applied to $P$ with the
  Alexandroff topology, provided one notes that a poset is
  well-quasi-ordered if and only if it is Noetherian in its
  Alexandroff topology. But the proof simplifies.

  $(a)\Rightarrow (c)$. Let $P'$ be an initial segment. By an already
  cited result of Hausdorff, $P'$ contains a well-founded cofinal
  subset $A$. Since $P$ has no infinite
  antichain, $P'$ has no infinite antichain; being well-founded it is
  w.q.o.

  $(c)\Rightarrow (b)$. Let $P'$ be an initial segment and $A$ be a
  cofinal subset of $P'$ which is w.q.o. Being w.q.o., $A$ is a finite
  union of ideals $I_1, \dots, I_q$. This is a basic result of the
  theory of w.q.o. \cite[Chapter VIII, Corollary, p.181]{birkhoff}.
  Indeed, as in the proof of implication $(iii)\Rightarrow (ii)$ of
  Theorem~\ref{theorem-main}, replacing ``closed'' by ``initial
  segment'' and ``irreducible closed'' by ``ideal'', if $A$ is not
  such, then, since $\mathbf {I}(A)$ is well-founded, there is a
  minimal member $A'\in \mathbf {I}(A)$ which is not a finite union of
  ideals. This $A'$ is irreducible, hence is an ideal: contradiction.

  Now $P'= \dc A= \dc I_1 \cup \dots \dc I_q$ and
  the set $\dc I_i$ are ideals of $P'$.

  $(b)\Rightarrow (a)$. Let $A$ be an antichain of $P$. An ideal $I$
  of $\dc A$ cannot contain more than one element of $A$. Since
  $\dc A$ is a finite union of ideals, $A$ is finite.
\end{proof}

A direct proof of $(a)\Rightarrow (b)$ can be obtained from a special
case of a result of Erd\H os and Tarski \cite{erdos-tarski}.  This
special case is a prototypal min-max result which has been overlooked,
and which states the following:
\begin{remark}[Erd\H os-Tarski]
  \label{rem:ET}
  If a poset $P$ contains no infinite up-independent set then there is
  a finite upper bound on the cardinality of up-independent sets. In
  this case, the maximum cardinality of up-independent sets, the least
  number of ideals whose union is $P$ and the least number of
  consistent sets whose union is $P$ are equal.
\end{remark}
We will dispense with the proof, as it will be a consequence of
Proposition~\ref{prop:ET:top} below.  Now, assuming $(a)$, we can
prove $(b)$ as follows.  Let $I$ be an initial segment of $P$.  Since
by $(a)$ there is no infinite antichain in $P$, $I$ does not contain
any infinite antichain either.  In the subposet $I$, the
up-independent subsets are all finite, since every up-independent
subset is an antichain.  Their maximum cardinality is then the least
number of ideals whose union is $I$, by Remark~\ref{rem:ET}.

Just as Theorem~\ref{theorem-main} generalizes
Theorem~\ref{thm:fraisse}, we will generalize Remark~\ref{rem:ET} to
the topological setting in Proposition~\ref{prop:ET:top}.

In the topological setting, we replace up-independence with
\emph{relative Hausdorffness}, defined as follows.  A subset $L$ of a
topological space $T:=(X, \mathcal F)$ is \emph{relatively Hausdorff}
if and only if for every two distinct points $x$ and $y$ of $L$, we
can find two disjoint open neighborhoods $U$ of $x$ and $V$ of $y$
\emph{in $X$} (not in the subspace $L$ itself, whence ``relative'').
In particular, $X$ is relatively Hausdorff in itself if and only if it
is Hausdorff.  Given a poset $P$ in its Alexandroff topology, every
point $x$ has a smallest open neighborhood, which is $\upc x$.  Then
$L \subseteq P$ is relatively Hausdorff if and only if for every two
distinct points $x$ and $y$ of $L$, $\upc x$ and $\upc y$ are disjoint
in $P$, if and only if $L$ is up-independent.

\begin{remark}
  \label{rem:relH}
  Every relatively Hausdorff subset $L$ of a topological space $X$ is
  Hausdorff in the subspace topology.  However, the reverse
  implication fails in general.  In order to see this, let us consider
  any space $X^\top$ obtained by adding an element $\top$ to a
  Hausdorff space $X$ and requiring that the non-empty open subsets of
  $X^\top$ are those of the form $U \cup \{\top\}$, where $U$ ranges
  over the open subsets of $X$.  Then $X$ is Hausdorff in the subspace
  topology from $X^\top$, but is not relatively Hausdorff in $X^\top$
  unless $X$ has at most one element.
\end{remark}

We also replace consistency by hyperconnectedness.  A
\emph{hyperconnected} space is a non-empty space in which any two
non-empty open sets intersect.
In a poset $P$ with the Alexandroff topology, a subspace $Q$ is
hyperconnected in the induced topology if and only of the smallest
open neighborhoods $\upc x$ and $\upc y$ of any two points $x$ and $y$
intersect, if and only if $Q$ is consistent.  Hence the following has
Remark~\ref{rem:ET} as a special case.

\begin{proposition}[Min-max]
  \label{prop:ET:top}
  In a topological space $T:=(X, \mathcal F)$ with no infinite
  relatively Hausdorff subset, there is a finite upper bound on the
  cardinalities of relatively Hausdorff subsets.  In this case, the
  following numbers exist and are equal:
  \begin{enumerate}[label=(\alph*)]
  \item the maximum cardinality of relatively Hausdorff subsets of
    $X$;
  \item the maximum number of pairwise disjoint non-empty open subsets
    of $X$.
  \item the least number of irreducible closed subsets whose union is
    $X$;
  \item the least number of hyperconnected subspaces of $X$ whose
    union is $X$.
  \end{enumerate}
\end{proposition}
\begin{proof}
  We write $\varphi$ for closure in $T$.  We proceed by making a
  series of observations.

  \emph{Observation 1.}  Given any relatively Hausdorff subset $Y$ of
  $X$, say of cardinality $n$, we can find $n$ pairwise disjoint
  non-empty open subsets of $X$.  In order to see this, let us write
  $Y$ as $\{y_1, \cdots, y_n\}$.  Since $Y$ is relatively Hausdorff,
  for every pair of indices $i$, $j$ with $1\leq i<j \leq n$, there
  are disjoint open sets $U_{ij}$ and $V_{ij}$ such that
  $y_i \in U_{ij}$ and $y_j \in V_{ij}$.  For each $y := y_i$ in $Y$,
  let $U_y := \bigcap_{j>i} U_{ij} \cap \bigcap_{j<i} V_{ji}$.  Then
  $y \in U_y$, and the sets $U_y$ are pairwise disjoint.

  \emph{Observation 2.} Given $n$ pairwise disjoint non-empty open
  subsets $U_1$, \ldots, $U_n$ of $X$, there is a relatively Hausdorff
  subset $Y$ of $X$ of cardinality $n$: we simply pick one point from
  each $U_i$.

  \emph{Observation 3.}  Every irreducible closed subset of $X$ is
  hyperconnected.  Indeed, let $C$ be irreducible closed in $X$.  We
  consider two non-empty open subsets of $C$, necessarily of the form
  $U \cap C$ and $V \cap C$, where $U$ and $V$ are open in $X$.  Then
  $C \setminus U$ and $C \setminus V$ are proper closed subsets of
  $C$, and since $C$ is irreducible,
  $C \neq (C \setminus U) \cup (C \setminus V)$; in other words,
  $U \cap V \neq \emptyset$.

  \emph{Observation 4.}  The closure (in $T$) of every hyperconnected
  subspace is irreducible closed.  Indeed, let $U$ be a hyperconnected
  subspace of $T$, and let us assume that we can write $\varphi (U)$
  as the union of two proper closed subsets $C_0$ and $C_1$.  Since
  $C_0$ and $C_1$ are proper, $V_0 := X \setminus C_0$ and
  $V_1 := X \setminus C_1$ both intersect $\varphi (U)$.  Hence they
  both intersect $U$.  It follows that $U \cap V_0$ and $U \cap V_1$
  are non-empty open subsets of $U$, and since $U$ is hyperconnected,
  they must intersect.  But
  $(U \cap V_0) \cap (U \cap V_1) = U \setminus (C_0 \cup C_1)
  \subseteq \varphi (U) \setminus (C_0 \cup C_1) = \emptyset$, which
  is a contradiction; so $\varphi (U)$ is irreducible.

  \emph{Observation 5.}  Every non-empty subset $U$ of $X$ contains a
  hyperconnected open subset.  Otherwise, $U$ is bad, where we call
  \emph{bad} any non-empty open subset of $X$ with no hyperconnected
  open subset.  We form an infinite binary tree $\mathcal T$ as
  follows, whose vertices $s$ are all labeled by bad sets $V_s$.
  (Here $s$ ranges over the finite 0-1 strings, and we write
  $\epsilon$ for the empty string, $s0$ and $s1$ for the string $s$
  with $0$, resp.\ $1$, appended to its end.)  Its root is
  $V_\epsilon := U$.  Given any vertex $s$ labeled with a bad set
  $V_s$, since $V_s$ itself is not hyperconnected, there are two
  non-empty open subsets $V_{s0}$ and $V_{s1}$ of $V$ whose
  intersection is empty.  The sets $V_{s0}$ and $V_{s1}$ are
  themselves bad, since they are included in the bad set $V_s$.  We
  label the two successors $s0$ and $s1$ of $s$ by $V_{s0}$ and by
  $V_{s1}$ respectively.  Now that $\mathcal T$ has been built, the
  collection $\{V_{0^n1} : n \in \N\}$ consists of infinitely many
  pairwise disjoint non-empty open sets.  We pick $x_n \in V_{0^n1}$
  for each $n \in \N$, and then $\{x_n : n \in \N\}$ is an infinite
  relatively Hausdorff subset of $X$, since the open sets $V_{0^n1}$
  are pairwise disjoint; and this is impossible.  This shows that $U$
  cannot be bad, proving the claim.

  \emph{Observation 6.}  Let us call \emph{sieve} any collection $S$
  of pairwise disjoint hyperconnected open sets.  We note that every
  sieve is finite, by Observation~2.  Every sieve $S$ is included in a
  maximal sieve $S_{\max}$, with respect to inclusion, by Zorn's
  Lemma, and if $S_{\max}$ is any maximal sieve, say of cardinality
  $n$, then we can find a minimal cover of $X$ by irreducible closed
  sets, of cardinality $n$.

  This is proved as follows.   We claim that
  $\bigcup_{U \in S_{\max}} U$ is dense in $X$.  Otherwise
  $X \setminus \varphi (\bigcup_{U \in S_{\max}} U)$ would be a
  non-empty open set, hence it would contain a hyperconnected open set
  $V$ (by Observation~5), and then $S_{\max} \cup \{V\}$ would be a
  strictly larger sieve, contradicting maximality.

  For every $U \in S_{\max}$, $U$ is hyperconnected, so $\varphi (U)$
  is irreducible closed, by Observation~4.  Since
  $\bigcup_{U \in S_{\max}} U$ is dense in $X$ and since $\varphi$ is
  topological, $X$ is covered by the $n$ irreducible closed sets
  $\varphi (U)$, $U \in S_{\max}$.

  The cover $\{\varphi (U) \mid U \in S_{\max}\}$ is minimal, namely
  removing any element would fail to produce a cover.  Indeed, let $U$
  be any element of $S_{\max}$, and let us pick a point $x$ from $U$.
  Then
  $x \not\in \bigcup_{V \in S_{\max} \setminus \{U\}} \varphi (V)$,
  since otherwise $U$ would intersect $\varphi (V)$ for some
  $V \in S_{\max} \setminus \{U\}$, hence also $V$ itself.

  \emph{Observation 7.}  Any two minimal finite covers
  $C_1, \cdots, C_m$ and $C'_1, \cdots, C'_n$ of $X$ by irreducible
  closed subsets have the same cardinality.

  For each $i \in \{1, \cdots, m\}$,
  $C_i \subseteq X = \bigcup_{j=1}^n C'_j$.  Since $C_i$ is
  irreducible, is is easy to see that $C_i \subseteq C'_j$ for some
  $j \in \{1, \cdots, n\}$.  We pick one such $j$ and call it $f (i)$;
  therefore $f$ is a map from $\{1, \cdots, m\}$ to $\{1, \cdots, n\}$
  such that $C_i \subseteq C'_{f (i)}$ for every
  $i \in \{1, \cdots, m\}$.  Similarly, there is a map
  $g \colon \{1, \cdots, n\} \to \{1, \cdots, m\}$ such that
  $C'_j \subseteq C_{g (j)}$ for every $j \in \{1, \cdots, n\}$.  For
  every $i$, we have
  $C_i \subseteq C'_{f (i)} \subseteq C_{g (f (i))}$.  If
  $i \neq g (f (i))$, then removing $C_i$ from the list
  $C_1, \cdots, C_m$ would still produce a cover of $X$ by irreducible
  closed subsets, contradicting minimality; so $i=g(f (i))$, for every
  $i \in \{1, \cdots, m\}$.  Similarly, $j = f (g (j))$ for every
  $j \in \{1, \cdots, n\}$.  Therefore $f$ and $g$ are mutually
  inverse, and $m=n$.
  
  \emph{Observation 8.}  Given any minimal finite cover of $X$ by
  irreducible closed subsets $C_1$, \ldots, $C_n$, there is a
  relatively Hausdorff subset of $X$ of cardinality $n$.  Indeed, for
  every $i \in \{1, \cdots, n\}$, by minimality
  $C_i \not\subseteq \bigcup_{k \neq i} C_k$, so we can pick a point
  $x_i \in C_i \setminus \bigcup_{k \neq i} C_k$.  The subset
  $Y := \{x_1, \cdots, x_n\}$ is relatively Hausdorff, since for all
  $i < j$, $x_i$ and $x_j$ are separated by the disjoint open sets
  $X \setminus \bigcup_{k \neq i} C_k$ and
  $X \setminus \bigcup_{k \neq j} C_k$ in $X$.
  
  Those observations allow us to prove the proposition as follows.  By
  Observations~1 and~2, the numbers that we can obtain as
  cardinalities of relatively Hausdorff subsets of $X$ or as
  cardinalities of collections of pairwise disjoint non-empty subsets
  of $X$ are the same.  By Observation~5, those are also the possible
  cardinalities of sieves, and by Observation~6 they are bounded from
  above; let $n_{\max}$ be their maximum.  Observation~6 also tells us
  that $n_{\max} \leq n_{\min}$ (and that $n_{\min} < \infty$), where
  $n_{\min}$ is the cardinality of some minimal cover of $X$ by
  irreducible closed sets.  By Observation~7, $n_{\min}$ is,
  equivalently, the least number of irreducible closed sets needed to
  cover $X$.  Observation~8 tells us that $n_{\min} \leq n_{\max}$, so
  $n_{\min} = n_{\max}$.

  Finally, if $n$ hyperconnected sets suffice to cover $X$, then their
  closures are irreducible closed and cover $X$ by Observation~4, and
  conversely, if $n$ irreducible closed sets cover $X$, then they are
  hyperconnected by Observation~3, so $n_{\min}$ is also the least
  number of hyperconnected sets covering $X$.
\end{proof}

We have already mentioned that the Erd\H os-Tarski result
(Remark~\ref{rem:ET}) is the special case of the equality of the
numbers mentioned in items~$(a)$ and $(d)$ of
Proposition~\ref{prop:ET:top}, when $X$ is a poset $P$ with its
Alexandroff topology.  While the Erd\H os-Tarski result implies the
$(a) \Rightarrow (b)$ direction of Theorem~\ref{thm:fraisse}, it is
not clear that Proposition~\ref{prop:ET:top} would entail the
$(i) \Rightarrow (ii)$ direction of Theorem~\ref{theorem-main}.  In
analogy with the setting of posets, where we had used the fact that
every up-independent subset is an antichain, we would need to say that
every relatively Hausdorff subset is discrete, but this is clearly
wrong.  (Consider any non-discrete Hausdorff space.)  However, this
does hold in FAC spaces, as we now see.  Since the proof uses
Proposition~2.7 of \cite{goubault-larrecq}, which relies on
Theorem~2.1 there, which itself is the equivalence of $(i)$ and $(ii)$
of Theorem~\ref{theorem-main}, it would be a fallacy to use it to
derive $(i) \Rightarrow (ii)$, though.
%
%
But we obtain the following, which we offer as our conclusion.  (A
KC-space is a space in which every compact subset is closed.)
\begin{theorem}
  \label{thm:haus:discr}
  A topological space $T:=(E, \mathcal F)$ is a FAC space, namely
  satisfies the equivalent conditions~$(i)$, $(ii)$ or $(iii)$ of
  Theorem~\ref{theorem-main}, or A.H. Stone's equivalent
  conditions~$(iv)$--$(vii)$ (see Section~\ref{sec:other-char}), if
  and only if it satisfies any of the following equivalent conditions:
  \begin{enumerate}
  \item[(viii)] No infinite subspace of $T$ is both sober and $T_1$;
  \item[(ix)] No infinite subspace of $T$ is a KC-space;
  \item[(x)] No infinite subspace of $T$ is Hausdorff;
  \item[(xi)] For no subspace $X$ of $T$ is there any infinite subset
    of $X$ that is relatively Hausdorff in $X$.
  \end{enumerate}
  Then the min-max conditions of Proposition~\ref{prop:ET:top} hold on
  every subspace $X$ of $T$.
\end{theorem}
\begin{proof}
  Let $\varphi$ be closure in $T$.  We claim that if
  $T = (E, \mathcal F)$ is a FAC space, then every subspace $L$ of $E$
  is a FAC space.  Indeed, given any infinite subset $A$ of $L$, $A$
  is not discrete in $E$ since $E$ is FAC, so there is a point $x$ in
  $A$ such that $x \in \varphi (A \setminus \{x\})$.  Then
  $x \in \varphi (A \setminus \{x\}) \cap L = \varphi_{\restriction L}
  (A \setminus \{x\})$, showing that $A$ is not discrete in $L$.

  Proposition~2.7 of \cite{goubault-larrecq} states that, for a FAC
  space, it is equivalent to be sober and $T_1$, or a KC-space, or
  Hausdorff, or finite and discrete.  Hence, if $T$ is a FAC space,
  then $(viii)$--$(x)$ hold.  Since every relatively Hausdorff subset
  is Hausdorff as a subspace (see Remark~\ref{rem:relH}), the
  implication $(x) \Rightarrow (xi)$ follows.  The implications
  $(viii) \Rightarrow (x)$ and $(ix) \Rightarrow (x)$ follow from the
  fact that every Hausdorff space is sober, $T_1$, and a KC-space.
  Finally, if $(xi)$ holds, then Proposition~\ref{prop:ET:top}
  applies, and item~$(c)$ of that proposition (applied to the case of
  a closed subspace $X$) implies item~$(ii)$ in
  Theorem~\ref{theorem-main}, showing that $T$ is a FAC space.
\end{proof}

\section*{Acknowledgments}
\label{sec:acknowledgments}

We thank the anonymous reviewer, and we are grateful to Akira Iwasa,
who found a mistake in a previous version of
Theorem~\ref{thm:haus:discr}.


\end{document}